\documentclass[12pt,a4paper,reqno]{amsart}
\usepackage{amssymb}
\usepackage{amscd}
\usepackage{enumerate}
\usepackage{graphicx}
\usepackage{siunitx}
\usepackage{tikz-cd}
\usepackage{color}
\usetikzlibrary{arrows}
\numberwithin{equation}{section}

\usepackage{mathabx}

\usepackage{mathtools}
\usepackage[tableposition=top]{caption}
\usepackage{booktabs,dcolumn}

\DeclareFontFamily{OT1}{rsfs}{}
\DeclareFontShape{OT1}{rsfs}{n}{it}{<-> rsfs10}{}
\DeclareMathAlphabet{\mathscr}{OT1}{rsfs}{n}{it}

\theoremstyle{plain}

\newtheorem{theorem}{Theorem}[section]
\newtheorem{proposition}[theorem]{Proposition}
\newtheorem{lemma}[theorem]{Lemma}

\newtheorem{claim}[theorem]{Claim}

\theoremstyle{definition}

\newcommand\Z{\mathbb{Z}}
\newcommand\N{\mathbb{N}}

\newcommand\eps{\varepsilon}

\usepackage{algorithm, algorithmic}

\begin{document}

\title[$L^1$ pointwise ergodic theorem fails for $F_2$]{Failure of the $L^1$ pointwise and maximal ergodic theorems for the free group}

\author{Terence Tao}
\address{UCLA Department of Mathematics, Los Angeles, CA 90095-1555.}
\email{tao@math.ucla.edu}



\begin{abstract}  Let $F_2$ denote the free group on two generators $a,b$.  For any measure-preserving system $(X, {\mathcal X}, \mu, (T_g)_{g \in F_2})$ on a finite measure space $X = (X,{\mathcal X},\mu)$, any $f \in L^1(X)$, and any $n \geq 1$, define the averaging operators
$${\mathcal A}_n f(x) := \frac{1}{4 \times 3^{n-1}} \sum_{g \in F_2: |g| = n} f( T_g^{-1} x ),$$
where $|g|$ denotes the word length of $g$.
We give an example of a measure-preserving system $X$ and an $f \in L^1(X)$ such that the sequence ${\mathcal A}_n f(x)$ is unbounded in $n$ for almost every $x$, thus showing that the pointwise and maximal ergodic theorems do not hold in $L^1$ for actions of $F_2$.  This is despite the results of Nevo-Stein and Bufetov, who establish pointwise and maximal ergodic theorems in $L^p$ for $p>1$ and for $L \log L$ respectively, as well as an estimate of Naor and the author establishing a weak-type $(1,1)$ maximal inequality for the action on $\ell^1(F_2)$.  Our construction is a variant of a counterexample of Ornstein concerning iterates of a Markov operator.
\end{abstract}

\maketitle


\section{Introduction}

Let $F_2$ denote the free non-abelian group on two generators $a,b$.  Define a \emph{reduced word} to be a word with letters in the alphabet $\{a,b,a^{-1},b^{-1}\}$ in which $a,a^{-1}$ and $b,b^{-1}$ are never adjacent, and for each $g \in F_2$, define the \emph{word length} $|g|$ of $g$ to be the length of the unique reduced word that produces $g$.    We let $F_2^2$ denote the index $2$ subgroup of $F_2$ consisting of $g \in F_2$ with even word length.

Define a \emph{$F_2$-system} to be a quadruple $(X, {\mathcal X}, \mu, (T_g)_{g \in F_2})$, where $(X,{\mathcal X}, \mu)$ is a measure space with $0 < \mu(X) < \infty$, and $T_g \colon X \to X$ is a family of measure-preserving maps on $X$ for $g \in F_2$, with $T_1$ the identity and $T_g T_h = T_{gh}$ for all $g,h \in F_2$; in particular, the $T_g$ are bi-measurable with $T_g^{-1} = T_{g^{-1}}$.  One can of course normalise such systems to have total measure $1$ by dividing $\mu$ by $\mu(X)$, but (as we will eventually be gluing several systems together) it will be convenient not to always insist on such a normalisation.  As the free group $F_2$ has no relations, such a system can be prescribed by specifying two arbitrary invertible bi-measurable measure-preserving maps $T_a,T_b:X \to X$, and then defining $T_g$ for all other $g \in G$ in the obvious fashion.

We say that an $F_2$-system is \emph{$F_2$-ergodic} if all $F_2$-invariant measurable sets either have zero measure or full measure, and \emph{$F_2^2$-ergodic} if the same claim is true for $F_2^2$-invariant measurable sets. For any $f \in L^1(X) = L^1(X,{\mathcal X},\mu)$ and any $n \geq 1$, we define the averaging operators
$${\mathcal A}_n f(x) := \frac{1}{4 \times 3^{n-1}} \sum_{g \in F_2: |g| = n} f( T_g^{-1} x );$$
note that $4 \times 3^{n-1}$ is the number of reduced words of length $n$.  One can of course use symmetry to replace $T_g^{-1}$ by $T_g$ if desired.

The pointwise convergence of the operators ${\mathcal A}_n$ was studied by Nevo and Stein \cite{nevo} and Bufetov \cite{bufetov}, who (among other things) proved the following result:

\begin{theorem}[Pointwise ergodic theorem]\label{pet}  Let $(X,{\mathcal X},\mu,(T_g)_{g \in F_2})$ be an $F_2$-system.  If $\int_X |f| \log(2+|f|)\ d\mu < \infty$, then ${\mathcal A}_{2n} f$ converges pointwise almost everywhere (and in $L^1(X)$ norm) to an $F_2^2$-invariant function.  In particular, if $(X,{\mathcal X},\mu,(T_g)_{g \in F_2})$ is $F_2^2$-ergodic, then ${\mathcal A}_{2n} f$ converges pointwise almost everywhere and in $L^1$ to the constant $\frac{1}{\mu(X)} \int_X f\ d\mu$.
\end{theorem}

The restriction to even averages ${\mathcal A}_{2n}$, and the use of $F_2^2$ instead of $F_2$, can be seen to be necessary by considering the simple example in which $X$ is a two-element set $\{0,1\}$ (with uniform measure) and $T_a,T_b$ interchange the two elements $0,1$ of this set.  The original paper of Nevo and Stein \cite{nevo} established this theorem for $f \in L^p(X)$ for some $p>1$, by modifying the methods of Stein \cite{stein}.  The subsequent paper of Bufetov \cite{bufetov} used instead the ``Alternierende Verfahren'' of Rota \cite{rota} to cover the $L \log L$ case.  Both arguments also extend to several other group actions (see e.g. \cite{nevo-2}, \cite{fuji}, \cite{marg}), but for simplicity of exposition we shall focus only on the $F_2$ case.  We also remark that both arguments also give bounds on the associated maximal operator $f \mapsto \sup_n {\mathcal A}_n |f|$.  See also \cite{bowen-1}, \cite{bowen-2} for an alternate approach to pointwise ergodic theorems in $L^p$ and $L \log L$.

In \cite{nevo} the question was posed as to whether the above pointwise ergodic theorem extended to arbitrary $L^1(X)$ functions.  The main result of this paper answers this question in the negative:

\begin{theorem}[Counterexample]\label{main}  There exists an $F_2$-system $(X,{\mathcal X},\mu,(T_g)_{g \in F_2})$ and an $f \in L^1(X)$ such that $\sup_n |{\mathcal A}_{2n} f(x)| = \infty$ for almost every $x \in X$.  In particular, ${\mathcal A}_{2n} f(x)$ fails to converge to a limit as $n \to \infty$ for almost every $x \in X$.
\end{theorem}

As such, there is no pointwise ergodic theorem or maximal ergodic theorem in $L^1$ for actions of the free group $F_2$.  Our construction also applies to free groups $F_r$ on $r$ generators for any $r \geq 2$; we leave the modification of the arguments below to this more general case to the interested reader.  This result stands in contrast to the situation for the regular action of $F_2$ on $\ell^1(F_2)$, for which a weak-type (1,1) for the maximal operator was established by Naor and the author \cite[Theorem 1.5]{naor}.  Note that the estimate for $\ell^1(F_2)$ does not transfer to arbitrary $F_2$-systems due to the non-amenability of the free group $F_2$.

Because the sphere $\{ g \in F_2: |g| = n \}$ is a positive fraction of the ball $\{ g \in F_2: |g| \leq n \}$, the above result also holds if the average over spheres is replaced with an average over balls, or with regards to other minor variations of the spherical averaging operator such as $\frac{1}{2} {\mathcal A}_n + \frac{1}{2} {\mathcal A}_{n+1}$.  This negative result for averaging on balls stands in contrast with the situation for amenable groups, for which pointwise and maximal ergodic results in $L^1$ are established for suitable replacements of balls, such as tempered F{\o}lner sets; see \cite{lindenstrauss}.  On the other hand, if one considers the Ces\'aro means $\frac{1}{N} \sum_{n \leq N} {\mathcal A}_n$ of spherical averages on $F_2$-systems, then pointwise and maximal ergodic theorems in $L^1$ were established in \cite{nevo}.

Our construction is inspired by a well-known counterexample of Ornstein \cite{ornstein} demonstrating the failure of the maximal ergodic theorem in $L^1$ for iterates $P^n$ of a certain well-chosen self-adjoint Markov operator.  Roughly speaking, the function $f$ in Ornstein's example consists of many components $f_i$, each of which comes with a certain ``time delay'' that ensures that the dynamics of $P^n f_i$ only become significant after a significant period of time - in particular, long enough for the dynamics of other components of the function to have achieved ``mixing'' in the portion of $X$ where the most interesting portion of the dynamics of $P^n f_i$ takes place, allowing the amplitude of $f_i$ to be slightly smaller than would otherwise have been necessary to make $\sup_n P^n f$ large.  To adapt this construction to the setting of $F_2$-systems, we need to glue together various $F_2$-systems that have the capability to produce such a ``time delay''.  We will be able to construct such systems by basically taking an ``infinitely large ball'' in $F_2$, gluing the boundary of that ball to itself, and redefining the shift maps on the boundary appropriately.  Somewhat ironically, the positive results in Theorem \ref{pet} play a helpful supporting role in establishing the negative result in Theorem \ref{main}, by establishing the ``mixing'' referred to previously that is an essential part of Ornstein's construction.

\subsection{Acknowledgments}  The author is supported by NSF grant DMS-1266164 and by a Simons Investigator Award, and thanks Lewis Bowen for helpful discussions and corrections.

\section{Initial reductions}

We begin by reducing Theorem \ref{main} to the following more quantitative statement.

\begin{theorem}[Quantitative counterexample]\label{main-2}  Let $\alpha, \eps > 0$.  Then there exists an $F_2$-system $(X,{\mathcal X},\mu,(T_g)_{g \in F_2})$ and a non-negative function $f \in L^\infty(X)$, such that
$$ \| f \|_{L^1(X)} \leq \alpha \mu(X)$$
but such that
$$ \sup_n {\mathcal A}_{2n} f(x) \geq 1-\eps$$
for all $x \in X$ outside of a set of measure at most $\eps \mu(X)$.
\end{theorem}

Let us see how Theorem \ref{main-2} implies Theorem \ref{main}.  By dividing $\mu$ by $\mu(X)$ we may normalise $\mu(X)=1$ in Theorem \ref{main-2}.  Applying the above theorem with $\alpha=\eps=2^{-m}$, we can thus find for each natural number $m$, an $F_2$-system $(X_m, {\mathcal X}_m, \mu_m, (T_{g,m})_{g \in F_2})$ with $\mu_m(X_m)=1$, and a non-negative function $f_m \in L^\infty(X_m)$ such that
$$ \| f_m \|_{L^1(X_m)} \leq 2^{-m}$$
and
$$ \sup_n {\mathcal A}_{2n} f_m(x) \geq 1 - 2^{-m} $$
outside of a set of measure $2^{-m}$.

Let $(X, {\mathcal X}, \mu, (T_g)_{g \in F_2})$ be the product system, thus $X$ is the Cartesian product $X := \prod_m X_m$ with product $\sigma$-algebra ${\mathcal X} := \prod_m {\mathcal X}_m$, product probability measure $\mu := \prod_m \mu_m$, and product action $T_g := \biguplus_m T_{g,m}$.  Each $f_m \in L^\infty(X_m)$ then lifts to a function $\tilde f_m \in L^\infty(X)$ with
$$ \| \tilde f_m \|_{L^1(X)} \leq 2^{-m}$$
and
$$ \sup_n {\mathcal A}_{2n} \tilde f_m(x) \geq 1 - 2^{-m} \geq 1/2$$
outside of a set of measure $2^{-m}$.  If we then set $f := \sum_m m \tilde f_m$, then $f \in L^1(X)$, and from the pointwise inequality
$$ \sup_n {\mathcal A}_{2n} f(x) \geq m_0 \sup_n {\mathcal A}_{2n} \tilde f_m(x)$$
for all $m \geq m_0$ and the Borel-Cantelli lemma we see that $\sup_n {\mathcal A}_{2n} f(x)$ is larger than $m_0/2$ for almost every $x$ and any given $m_0$, which yields the claim.

It remains to prove Theorem \ref{main-2}.  In order to adapt the arguments of Ornstein \cite{ornstein}, we would like to interpret the averaging operators ${\mathcal A}_n$ as powers $P^n$ of a Markov operator $P$.  This is not true as stated, since we do not quite have the semigroup property ${\mathcal A}_n {\mathcal A}_m = {\mathcal A}_{n+m}$ (although ${\mathcal A}_n {\mathcal A}_m$ does contain a term of the form $\frac{3}{4} {\mathcal A}_{n+m}$).  However, as observed by Bufetov \cite{bufetov}, we can recover a Markov interpretation for ${\mathcal A}_n$ by lifting $X$ up to a four-fold cover $\tilde X$ that tracks the ``outward normal vector'' for the sphere.  More precisely, given an $F_2$-system $(X,{\mathcal X}, \mu, (T_g)_{g \in F_2})$, we define the lifted measure space $(\tilde X, \tilde {\mathcal X}, \tilde \mu)$ to be the product of $(X,{\mathcal X},\mu)$ and the four-element space $\{a,b,a^{-1},b^{-1}\}$ with the uniform probability measure; in particular $\tilde \mu(\tilde X) = \mu(X)$.  Let $\pi \colon \tilde X \to X$ be the projection operator $\pi(x,s) := x$; this induces a pushforward operator $\pi_* \colon L^1(\tilde X) \to L^1(X)$ and a pullback operator $\pi^* \colon L^1(X) \to L^1(\tilde X)$ by the formulae
$$ \pi_* \tilde f(x) := \frac{1}{4} \sum_{s \in \{a,b,a^{-1},b^{-1}\}} \tilde f(x,s)$$
and
$$ \pi^* f(x,s) := f(x)$$
for $f \in L^1(X)$ and $\tilde f \in L^1(\tilde X)$.  We also define the Markov operator $P \colon L^1(\tilde X) \to L^1(\tilde X)$ by
$$ P \tilde f(x,s) := \frac{1}{3} \sum_{s' \in \{a,b,a^{-1},b^{-1}\}: s' \neq s^{-1}} \tilde f( T_s^{-1} x, s' ).$$
One can view $P$ as the Markov operator associated to the Markov chain that for each unit time, moves a given point $(x,s)$ of $\tilde X$ to one of the three points $(T_{s'} x, s')$ with $s' \in \{a,b,a^{-1},b^{-1}\} \backslash \{s^{-1}\}$, chosen at random.  By writing the elements of $\{g \in F_2: |g|=n\}$ as reduced words of length $n$, one can easily verify the identity
$$ {\mathcal A}_n f = \pi_* P^n \pi^* f $$
for any $f \in L^1(X)$ and $n \geq 1$.  It thus suffices to show

\begin{theorem}[Quantitative counterexample, again]\label{main-3}  Let $\alpha, \eps > 0$.  Then there exists an $F_2$-system $(X,{\mathcal X},\mu,(T_g)_{g \in F_2})$ and a non-negative function $\tilde f \in L^\infty(\tilde X)$, such that
$$ \| \tilde f \|_{L^1(\tilde X)} \leq \alpha \mu(X)$$
but such that
$$ \sup_n \pi_* P^{2n} \tilde f(x) \geq 1-\eps$$
for all $x \in X$ outside of a set of measure at most $\eps \mu(X)$.
\end{theorem}

Indeed, by setting $f := 4 \pi_* \tilde f$, and noting the pointwise bound $\tilde f \leq \pi^* f$ and the identity $\|f\|_{L^1(X)} = 4 \| \tilde f \|_{L^1(\tilde X)}$, we obtain Theorem \ref{main-2} (after replacing $\alpha$ by $\alpha/4$).

For inductive reasons, we will prove a technical special case of Theorem \ref{main-3}, in which the $F_2$-system is of a certain ``good'' form, and the sequence $(P^n \tilde f)_{n \geq 0}$ is part of an ``ancient Markov chain'' $(\tilde f_n)_{n \in \Z}$ that extends to arbitrarily negative times as well as arbitrarily positive times.  More precisely, let us define a \emph{good system} to be an $F_2$-system $(X,{\mathcal X},\mu,(T_g)_{g \in F_2})$ which admits a decomposition $X = X_a \cup X_b \cup X_0$ into three disjoint sets $X_a,X_b,X_0$ admitting the following (somewhat technical) properties:

\begin{itemize}
\item[(i)]  (Measure)  One has $\mu(X_a) = \mu(X_b) = \frac{1}{4} \mu(X)$ and $\mu(X_0) = \frac{1}{2} \mu(X)$.  Furthermore, for any $0 \leq \kappa \leq \mu(X_b)$, one can find a measurable subset of $X_b$ of measure exactly equal to $\kappa$.
\item[(ii)]  (Invariance)  One has $T_a X_a = X_a$ and $T_b X_b = X_b$.  Also, one has the inclusions $T_a X_b \subset T_b X_a \cup T_b^{-1} X_a \subset X_0$.  
\item[(iii)]  (Ergodicity)  One can partition $X_a$ into  finitely many $T_a$-invariant components $X_{a,1},\dots,X_{a,m}$ of positive measure, such that $T_a^2$ is ergodic on each of the components $X_{a,i}$; that is, the only $T_a^2$-invariant measurable subsets of $X_{a,i}$ have measure either $0$ or $\mu(X_{a,i})$.
\item[(iv)]  (Generation)  One has $X = \bigcup_{g \in F_2} T_g X_{a,i}$ up to null sets for each $i=1,\dots,m$.  
\end{itemize}

Note that relatively few conditions are required on the dynamics on $X_b$; in particular, the ergodicity hypotheses on the system are located in the disjoint region $X_a$.  This will allow us to easily modify the dynamics on $X_b$ in order to ``glue'' two good systems together in Section \ref{glue}.

\begin{figure} [t]
\centering
\includegraphics{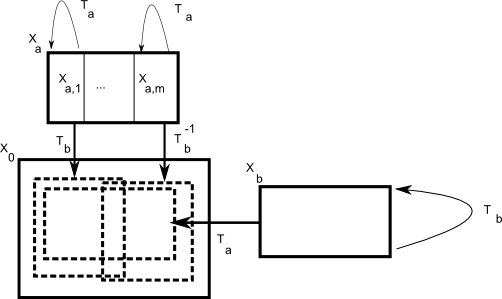}
\caption{A somewhat schematic depiction of a good system.  Only part of the action of $T_a$ and $T_b$ are displayed.}
\label{fig:good}
\end{figure}

See Figure \ref{fig:good}.
We will construct good systems in subsequent sections.  For now, we record one useful property of such systems:

\begin{lemma}[Pointwise ergodic theorem for good systems]\label{pet-good}  Every good system $(X,{\mathcal X},\mu,(T_g)_{g \in F_2})$ is $F_2^2$-ergodic.  In particular (by Theorem \ref{pet}), for any $f \in L^\infty(X)$, the averages ${\mathcal A}_{2n} f$ converge pointwise almost everywhere and in $L^1$ norm to $\frac{1}{\mu(X)} \int_X f\ d\mu$.  Furthermore, for any $\tilde f \in L^\infty(\tilde X)$, $P^{2n} \tilde f$ converge pointwise almost everywhere and in $L^1$ norm to $\frac{1}{\mu(X)} \int_{\tilde X} \tilde f\ d\tilde \mu$.
\end{lemma}

\begin{proof}  Let $f \in L^\infty(X)$ be an $F_2^2$-invariant function; to establish $F_2^2$-ergodicity, it will suffice to show that $f$ is constant almost everywhere.  As $f$ is $T_a^2$-invariant, we see from Axiom (iii) that $f$ is constant almost everywhere on each $X_{a,i}$.
 Since $F_2 = F_2^2 \cup F_2^2 a$, we see from Axiom (iv), the $T_a$-invariance of $X_{a,i}$, and the $F_2^2$-invariance of $f$ that $f$ is constant almost everywhere on $X$, as required.  The final claim does not quite follow from Theorem \ref{pet}, but is immediate from \cite[Proposition 1]{bufetov}.
\end{proof}

For any $\alpha > 0$, let $P(\alpha)$ denote the following claim:

\begin{claim}[$P(\alpha)$]\label{clam}  For any $\eps > 0$, there exists a good system $(X,{\mathcal X},\mu,(T_g)_{g \in F_2})$ with associated decomposition $X = X_a \cup X_b \cup X_0$, and a sequence of non-negative functions $\tilde f_n \in L^\infty(\tilde X)$ for $n \in \Z$ with the following properties:
\begin{itemize}
\item[(v)]  (Ancient Markov chain) $\tilde f_{n+1} = P \tilde f_n$ for all $n \in \Z$.  Equivalently, one has $\tilde f_{n+m} = P^m \tilde f_n$ for all $n \in \Z$ and $m \in \N$.  In particular, $\|\tilde f_n\|_{L^1(\tilde X)}$ is independent of $n$.
\item[(vi)]  (Size)  One has $\|\tilde f_n\|_{L^1(\tilde X)} = \alpha \mu(X)$ for some $n \in \Z$ (and hence for all $n \in \Z$).
\item[(vii)]  (Early support)  $\tilde f_n$ is supported in $\tilde X_0$ for all negative $n$.  Furthermore, there exists a finite $A>0$ such that $\tilde f_n$ is supported in a set of measure at most $A 3^n \mu(X)$ for all negative $n$.
\item[(viii)]  (Large maximum function)  We have
$$ \sup_{n \in \Z} \pi_* \tilde f_{2n}(x) \geq 1-\eps$$
for all $x \in X$ outside of a set of measure at most $\eps \mu(X)$.
\end{itemize}
\end{claim}

Note that our sequence $\tilde f_n$ is \emph{ancient} in the sense that it extends to arbitrary negative times $n \to -\infty$ as well as to arbitrary positive times $n \to \infty$.  This will be essential in order to set up suitable ``time delays'' in our arguments in later sections.  One can informally think of the $\tilde f_n$ as the (normalised) distribution at time $n$ of an ancient Markov process that starts from an infinitely small location deep inside $\tilde X_0$ at infinite negative time $n=-\infty$, and only escapes $\tilde X_0$ at or after time $n=0$, and which covers most of $X$ with density roughly $1$ or more at some point in time (but crucially, different regions of $X$ may be covered in this fashion at different times).

Observe that if $P(\alpha)$ holds for an arbitrarily small set of $\alpha>0$, and $\eps>0$ is arbitrary, then from axioms (vii), (viii), one has for any $N$ that
$$ \sup_{n \geq -2N} \pi_* f_{2n}(x) \geq 1-\eps$$
for all $x \in X$ outside of a set of measure at most $(\eps + \frac{9}{8} A 3^{-2N}) \mu(X)$.  Taking $N$ large enough (depending on $\eps$, $A$) and setting $\tilde f:= \tilde f_{-2N}$, we obtain Theorem \ref{main-3} (after adjusting $\eps$ as necessary).  It thus suffices to show that $P(\alpha)$ holds for arbitrarily small $\alpha>0$.  This will be accomplished using the following two key theorems (the second of which being a variant of \cite[Lemma 4]{ornstein}):

\begin{theorem}[Initial construction]\label{initial}  The claim $P( 1 )$ is true.
\end{theorem}

\begin{theorem}[Iteration step]\label{iterate}  Suppose that $P(\alpha)$ holds for some $0 < \alpha \leq 1$.  Then $P( \alpha (1-\frac{\alpha}{4}) )$ is true.
\end{theorem}

From Theorem \ref{initial} and Theorem \ref{iterate} we see that the infimum of all $0 < \alpha \leq 1$ for which $P(\alpha)$ holds is zero, and the claim follows.  Thus it suffices to establish Theorem \ref{initial} and Theorem \ref{iterate}.  This will be accomplished in the next two sections.

\section{The initial construction}

We now prove Theorem \ref{initial}.  We will in fact construct an example of a good system $(X, {\mathcal X}, \mu, (T_g)_{g \in F_2})$ and functions $\tilde f_n$ which witness $P(1)$ for all $\eps>0$ at once.

We begin by constructing an appropriate measure space $(X, {\mathcal X}, \mu)$.  For each integer $n$, let $Y_n$ denote the space of half-infinite reduced words $(s_m)_{m \geq n} = s_n s_{n+1} s_{n+2} \dots$, in which each of the $s_i$ are drawn from the alphabet $\{a,b,a^{-1},b^{-1}\}$ and $a,a^{-1}$ and $b,b^{-1}$ are never adjacent.  We give this space the product $\sigma$-algebra ${\mathcal Y}_n$ (that is, the minimal $\sigma$-algebra for which the coordinate maps $(s_m)_{m \leq n} \mapsto s_m$ are all measurable).  By the Kolmogorov extension theorem, we may construct a probability measure $\mu_n$ on $Y_n$ such that each finite reduced subword $s_n \dots s_{n+k}$ for $k \geq 0$ occurs as an initial segment with probability $\frac{1}{4 \times 3^k}$; one can view this measure as the law of the random half-infinite reduced word constructed by choosing $s_n$ uniformly at random from $\{a,b,a^{-1},b^{-1}\}$, then recursively selecting $s_{n+i+1}$ for $i=0,1,2,\dots$ to be drawn uniformly from $\{a,b,a^{-1},b^{-1}\} \backslash \{s_{n+i}^{-1}\}$.

The disjoint union $Y := \biguplus_{n \in \Z} Y_n$ of the $Y_n$ admits an action $(S_g)_{g \in F_2}$ of $F_2$, with the action $S_s$ of a generator $s \in \{a,b,a^{-1},b^{-1}\}$ defined by setting
$$ S_s ( s_n s_{n+1} s_{n+2} \dots ) := s s_n s_{n+1} s_{n+2} \dots \in Y_{n-1}$$
for $s_n s_{n+1} s_{n+2} \dots \in Y_n$ and $s \in \{a,b,a^{-1},b^{-1}\} \backslash s_n$, and
$$ S_s ( s_n s_{n+1} s_{n+2} \dots ) := s_{n+1} s_{n+2} \dots \in Y_{n-1}$$
for $s_n s_{n+1} s_{n+2} \dots \in Y_n$ and $s = s_n^{-1}$; thus $S_g$ is the operation of formal left-multiplication by $g$, after reducing any non-reduced words.  If we give $Y$ the measure $\mu_Y := \sum_{n \in \Z} 3^{-n} \mu_n$, then one can easily verify that this action is measure-preserving.  Unfortunately, $\mu_Y$ is an infinite measure due to the contribution of the negative $n$, and so this space is not quite suitable for our needs.  Instead, we shall work with a certain subquotient of $Y$, defined as follows.

\begin{figure} [t]
\centering
\includegraphics{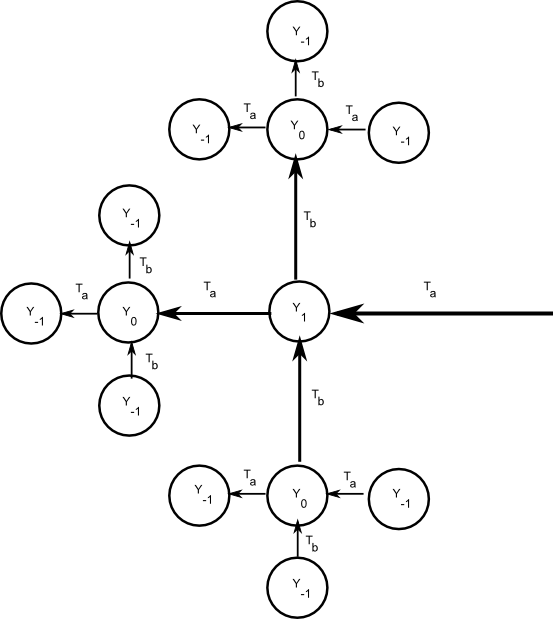}
\caption{A fragment of the infinite measure space $Y$.  The centre disk represents a portion of $Y_1$ consisting of reduced words $s_1 s_2 \dots$ with initial letter $s_1 = a$.  The remaining disks are images of this disk under shifts by various elements of $F_2$, and all have equal measure with respect to $\mu_Y$.  This image should be compared with the infinite tree that is the Cayley graph of $F_2$.}
\label{fig:ysp}
\end{figure}

Firstly, we restrict $Y$ to the space $\biguplus_{n \geq 0} Y_n = \biguplus_{n \geq 1} Y_n \uplus Y_0$, which can be thought of as a suitably rescaled limit of an ``infinitely large ball'' in $F_2$, with $Y_0$ being the ``boundary'' of this ball, and the $Y_n$ lying increasingly deeper in the ``interior'' of the ball as $n$ increases (see Figure \ref{fig:ysp}).  This makes the shift maps $S_s$, $s \in \{a,b,a^{-1},b^{-1}\}$ partially undefined on the $Y_0$ boundary, but we will fix this later by redefining these maps on (a quotient) of $Y_0$.  Next, we introduce a reflection operation $x \mapsto \overline{x}$ on the boundary $Y_0$ by mapping
$$ \overline{s_0 s_1 s_2 \dots} := s_0^{-1} s_1^{-1} s_2^{-2} \dots.$$
It is clear that this map preserves the measure $\mu_0$.  If we then form the quotient space $Y_0/\sim := \{ \{x,\overline{x}\}: x \in Y_0 \}$, we can obtain a probability measure $\mu_0/\sim$ on $Y_0/\sim$ by pushing forward the probability measure $\mu_0$ under the quotient map.  We observe that $Y_0/\sim$ splits into two components of equal measure $1/2$, namely $((S_a Y_1 \cap Y_0) \cup (S_a^{-1} Y_1 \cap Y_0))/\sim$ and $((S_b Y_1 \cap Y_0) \cup (S_b^{-1} Y_1 \cap Y_0))/\sim$, noting that the sets $S_a Y_1 \cap Y_0, S_a^{-1} Y_1 \cap Y_0$ are disjoint reflections of each other, and similarly for $S_b Y_1 \cap Y_0, S_b^{-1} Y_1 \cap Y_0$.

We then define $X$ to be the quotient space $\biguplus_{n \geq 1} Y_n \uplus (Y_0/\sim)$ with measure $\mu := \sum_{n \geq 1} 3^{-n} \mu_n + \frac{1}{2} (\mu_0/\sim)$, thus
$$ \mu(X) = \sum_{n \geq 1} 3^{-n} + \frac{1}{2} = 1.$$
We set $X_0 := \biguplus_{n \geq 1} Y_n$, $X_a := ((S_b Y_1 \cap Y_0) \cup (S_b^{-1} Y_1 \cap Y_0))/\sim$, and $X_b := ((S_a Y_1 \cap Y_0) \cup (S_a^{-1} Y_1 \cap Y_0))/\sim$.  Thus
$$ \mu(X_0) = \sum_{n \geq 1} 3^{-n}  = \frac{1}{2}$$
and $\mu(X_a) = \mu(X_b) = \frac{1}{4}$.   One can think of $X_0$ as the ``interior'' of $X$, with $X_a$ and $X_b$ being two equally sized pieces of the ``boundary'' of $X_0$.  Also, $X_a$, $X_b$ are Cantor spaces (and $\mu$ is a Cantor measure on such spaces), and so one can easily construct measurable subsets of $X_b$ of arbitrary measure between $0$ and $\mu(X_b)$.  Thus Axiom (i) is satisfied.  Also, one can easily create a measure-preserving invertible map $T^0_a \colon X_a \to X_a$ such that $(T^0_a)^2$ is ergodic on $X_a$; this can be done for instance by identifying $X_a$ (which is an atomless standard probability space) as a measure space (up to null sets) with the unit circle with Haar measure, and then setting $T^0_a$ to be an irrational translation map.

We now define the shifts $T_a \colon X \to X$ and $T_b \colon X \to X$ as follows.

\begin{enumerate}
\item If $x \in X_0$, then $T_a x$ is defined to be $S_a x$ projected onto $X$, and $T_b x$ is similarly defined to be $S_b$ projected onto $X$.  (The projection is only necessary of course if $S_a x$ or $S_b x$ lands in $Y_0$.)
\item If $x \in X_a$, then $T_a x := T_a^0 x$.  If instead $x \in X_b$, $T_a x$ is defined to be $S_a x' \in Y_1$, where $x' \in S_a^{-1} Y_1 \cap Y_0$ is the lift of $x$ to $S_a^{-1} Y_1 \cap Y_0$.
\item If $x \in X_b$, then $T_b x = x$.  If instead $x \in X_a$, $T_b x$ is defined to be $S_b x' \in Y_1$, where $x' \in S_b^{-1} Y_1 \cap Y_0$ is the lift of $x$ to $S_b^{-1} Y_1 \cap Y_0$.
\end{enumerate}

One then defines $T_g$ for the remaining $g \in F_2$ in the usual fashion.  In particular, one sees that for any $x$ in the interior $X_0$ and any $s \in \{a,b,a^{-1},b^{-1}\}$, $T_s x$ is equal to $S_s x$ projected onto $X$.  Informally, the shifts $T_s \colon X \to X$ for $s \in \{a,b,a^{-1},b^{-1}\}$ are inherited from the shifts $S_s \colon Y \to Y$ except for the boundary actions of $T_a,T_a^{-1}$ on $X_a$ and of $T_b, T_b^{-1}$ on $X_b$, which are given by $T_0^a$ (and its inverse) and the identity map respectively.  (There is nothing special about the identity map here; an arbitrary measure-preserving map on $X_b$ could be substituted here for our purposes.)

\begin{proposition}  $(X, {\mathcal X}, \mu, (T_g)_{g \in F_2})$ is a good system.
\end{proposition}

\begin{proof}
It is a routine matter to verify that $T_a, T_b$ are invertible and measure-preserving, so that  $(X, {\mathcal X}, \mu, (T_g)_{g \in F_2})$ is an $F_2$-system.  Axiom (i) was already verified.  For Axiom (ii), we note that $T_a X_b \subset Y_1 \subset (S_b Y_0 \cap Y_1) \cup (S_b^{-1} Y_0 \cap Y_1) = T_b X_0 \cup T_b^{-1} X_1$, as required.  We set $m=1$ and $X_{a,1} := X_a$, then Axiom (iii) is true from construction, and Axiom (iv) is also easily verified.
\end{proof}

It remains to construct a sequence $\tilde f_n$ of non-negative functions in $L^\infty(\tilde X)$ for each $n \in \Z$ obeying Axioms (v)-(viii) with $\alpha=1$.  For negative $n$, we define $\tilde f_n$ by setting
$$ \tilde f_n(x, s) := 4 \times 3^{-n}$$
whenever $x \in X$ and $s \in \{a,b,a^{-1},b^{-1}\}$ are such that $x \in Y_{-n}$ and $S_s x \in Y_{-n-1}$, and $\tilde f_n(x,s)=0$ otherwise.  These are clearly non-negative functions in $L^\infty(\tilde X)$ obeying Axiom (vii). It is routine to verify that $\tilde f_{n+1} = P \tilde f_n$ for all $n \leq -2$.  If we then define $\tilde f_n$ for non-negative $n$ by the formula
$$ \tilde f_n := P^{n+1} \tilde f_{-1}$$
then we have Axiom (v).  For negative $n$ we have
$$ \| \tilde f_n \|_{L^1(\tilde X)} = 1,$$
which gives Axiom (vi) (using Axiom (v) to extend to non-negative $n$).  Finally, from Lemma \ref{pet-good} we see that $\tilde f_n$ converges pointwise almost everywhere to $1$ as $n \to +\infty$, and so Axiom (vii) follows from Egorov's theorem.  This concludes the proof of Theorem \ref{initial}.

\section{The iteration step}\label{glue}

We now prove Theorem \ref{iterate}.  Let $0 < \alpha \leq 1$ be such that $P(\alpha)$ holds.  By Claim \ref{clam} (with $\eps$ replaced by $\eps/4$), and normalising $X$ to have measure $1$, we may find a good system $(X,{\mathcal X},\mu,(T_g)_{g \in F_2})$ with associated decomposition $X = X_a \cup X_b \cup X_0$ and measure $\mu(X)=1$, and a sequence of non-negative functions $\tilde f_n \in L^\infty(\tilde X)$ for $n \in \Z$ with the following properties:
\begin{itemize}
\item[(v)]  (Ancient Markov chain) $\tilde f_{n+1} = P \tilde f_n$ for all $n \in \Z$.  
\item[(vi)]  (Size)  One has $\|\tilde f_n\|_{L^1(\tilde X)} = \alpha$ for all $n \in \Z$.
\item[(vii)]  (Early support)  $\tilde f_n$ is supported in $\tilde X_0$ for all negative $n$.  Furthermore, there exists a finite $A>0$ such that $\tilde f_n$ is supported in a set of measure at most $A 3^n$ for all negative $n$.
\item[(viii)]  (Large maximum function)  We have
$$ \sup_{n \in \Z} \pi_* \tilde f_{2n}(x) \geq 1-\eps/4$$
for all $x \in X$ outside of a set of measure at most $\eps/4$.
\end{itemize}

It will suffice to construct a good system $(X',{\mathcal X}',\mu',(T'_g)_{g \in F_2})$ with associated decomposition $X' = X'_a \cup X'_b \cup X'_0$, Markov operator $P'$, and measure $\mu'(X')=2$, and a sequence of non-negative functions $\tilde f'_n \in L^\infty(\tilde X')$ for $n \in \Z$ with the following properties:
\begin{itemize}
\item[(v')]  (Ancient Markov chain) $\tilde f'_{n+1} = P' \tilde f'_n$ for all $n \in \Z$.  
\item[(vi')]  (Size)  One has $\|\tilde f'_n\|_{L^1(\tilde X')} = \alpha (2 - \frac{\alpha}{2})$ for all $n \in \Z$.
\item[(vii')]  (Early support)  $\tilde f'_n$ is supported in $\tilde X'_0$ for all negative $n$.  Furthermore, there exists a finite $A'>0$ such that $\tilde f'_n$ is supported in a set of measure at most $2 A' 3^n$ for all negative $n$.
\item[(viii')]  (Large maximum function)  We have
$$ \sup_{n \in \Z} \pi_* \tilde f'_{2n}(x') \geq 1-\eps$$
for all $x' \in X'$ outside of a set of measure at most $2\eps$.
\end{itemize}

We construct this system as follows.  First, from Axiom (viii) and Egorov's theorem, we may find a natural number $N$ such that
\begin{equation}\label{proj}
 \sup_{-N \leq n \leq N} \pi_* \tilde f_{2n}(x) \geq 1-\eps/3
\end{equation}
for all $x \in X$ outside of a set of measure at most $\eps/3$.  We let $0 < \kappa < 1/4$ be a small quantity depending on $\eps, N$ and the $\tilde f_n$ to be chosen later.  We will construct the good system $(X',{\mathcal X}',\mu',(T'_g)_{g \in F_2})$ to be two copies of $(X,{\mathcal X},\mu,(T_g)_{g \in F_2})$ glued together by a small amount of coupling, with the $\kappa$ parameter measuring the amount of coupling.  More precisely, we define the measure space $(X', {\mathcal X}', \mu')$ to be the product of $(X,{\mathcal X}, \mu)$ with the two-element set $\{1,2\}$ with counting measure.  Next, using Axiom (i), we can find a subset $E$ of $X_b$ of measure exactly $\kappa$.  We now define the shift maps $T'_a, T'_b\colon X' \to X'$ as follows.  The map $T'_a$ is a trivial lift of $T_a$, thus
$$ T'_a (x, i ) := (T_a x, i )$$
for $x \in X$ and $i \in \{1,2\}$.  The map $T'_b$ is an \emph{almost} trivial lift of $T_b$.  Namely, we define
$$ T'_b (x, i ) := (T_b x, i )$$
for $x \in X \backslash E$ and $i \in \{1,2\}$, but define
$$ T'_b (x, i ) := (T_b x, 3-i )$$
for $x \in E$ and $i \in \{1,2\}$; see Figure \ref{fig:dupl}.  Finally, we partition $X' = X'_a \cup X'_b \cup X'_0$ where $X'_a := X_a \times \{1,2\}$, $X'_b := X_b \times \{1,2\}$, $X'_0 := X_0 \times \{1,2\}$.  We then define $T'_g$ for the remaining $g \in F_2$ in the usual fashion.

\begin{figure} [t]
\centering
\includegraphics{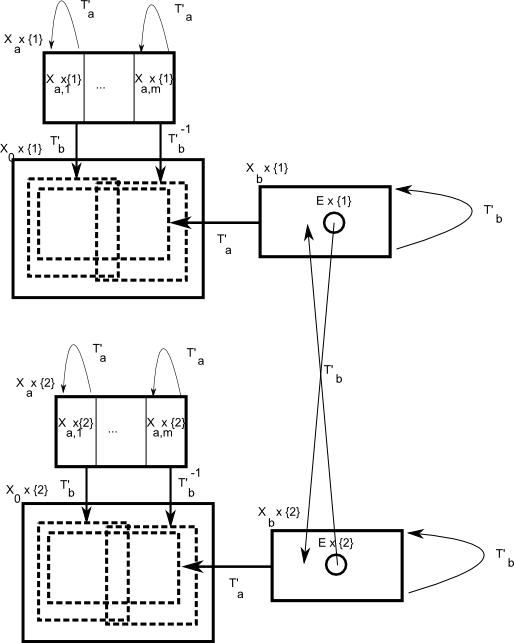}
\caption{The good system $(X',{\mathcal X}',\mu',(T'_g)_{g \in F_2})$, which is formed by gluing together two barely interacting copies of $(X,{\mathcal X},\mu,(T_g)_{g \in F_2})$.}
\label{fig:dupl}
\end{figure}

\begin{proposition}[Good system]  If $\kappa$ is sufficiently small, then $(X',{\mathcal X}',\mu',(T'_g)_{g \in F_2})$ is a good system with $\mu'(X') = 2$.
\end{proposition}

\begin{proof}  Axioms (i) and (ii) are easily verified, so we focus on verifying Axioms (iii) and (iv).

By Axiom (iii) for $X$, $X_a$ is partitioned into finitely many $T_a$-invariant components $X_{a,1},\dots,X_{a,m}$ of positive measure, each of which is $T_a^2$-ergodic.  This induces a partition of $X'_a$  into the $2m$ components $X_{a,1} \times \{1\}, \dots, X_{a,m} \times \{1\}, X_{a,1} \times \{2\}, \dots, X_{a,m} \times \{2\}$, and each of these components are clearly $T_a^2$-ergodic.

Now we verify Axiom (iv).  We need to show that $X' = \bigcup_{g \in F_2} T'_g ( X_{a,i} \times \{j\} )$ up to null sets for each $i=1,\dots,m$ and $j=1,2$.  Denote the right-hand side by $Y$, thus $Y$ is $F_2$-invariant and contains $X_{a,i} \times \{j\}$.  On the other hand, by Axiom (iv) for $X$ and the pigeonhole principle, there exists $g \in F_2$ such that $T_g X_{a,i}$ intersects $E$ in a set of positive measure.  We may assume that the word length $|g|$ of $g$ is minimal among all $g$ with this property, thus $T_h X_{a,i} \cap E$ is null whenever $|h| < |g|$.  From this we see that $T'_g (X_{a,i} \times \{j\})$ intersects $E \times \{j\}$ in a set of positive measure (since the dynamics of $T'$ are just a trivial lift of the dynamics of $T$ outside of $E \times \{1,2\}$).  From construction of $T'_b$, this implies that
$T'_{bg} (X_{a,i} \times \{j\})$ intersects $T_b E \times \{3-j\} \subset X_b \times \{3-j\}$ in a set of positive measure, and hence by Axiom (ii) the union of $T'_{b^{-1} abg} (X_{a,i} \times \{j\})$ and $T'_{babg} (X_{a,i} \times \{j\})$ intersects $X_a \times \{3-j\}$ in a set of positive measure; in particular, $Y$ intersects $X_a \times \{3-j\}$ in a set of positive measure.  As $Y$ is $(T'_a)^2$-invariant, we conclude from Axiom (iii) that $Y$ contains $X_{a,i'} \times \{3-j\}$ up to null sets for some $i'=1,\dots,m$.

Next, by another appeal to Axiom (iv) and the pigeonhole principle, we can find $g_{i,i'} \in F_2$ such that $T_{g_{i,i'}} X_{a,i'}$ and $X_{a,i}$ intersect in a set of positive measure.  Note that as there are only $m$ choices for $i'$, the word length of $g_{i,i'}$ can be bounded above, and the measure of $T_{g_{i,i'}} X_{a,i'} \cap X_{a,i}$ bounded below, by quantities independent of $\kappa$.  Because of this, we see that if $\kappa$ (and hence $E$) is small enough, then $T'_{g_{i,i'}} (X_{a,i'} \times \{3-j\})$ and $X_{a,i} \times \{3-j\}$ also intersect in a set of positive measure; thus $Y$ must intersect $X_{a,i} \times \{3-j\}$ in a set of positive measure, and hence by the $T_a^2$-ergodicity of $X_{a,i}$, $Y$ contains $X_{a,i} \times \{3-j\}$ up to null sets.  Since $Y$ already contained $X_{a,i} \times \{j\}$, we thus have $X_{a,i} \times \{1,2\}$ contained in $Y$ up to null sets.

Now for any $(x,j') \in X'$, we have from Axiom (iv) that $x = T_g y$ for some $y \in X_{a,i}$ and $g \in F_2$.  This implies that $(x,j') = T'_g (y,j'')$ for some $j'' \in \{1,2\}$, and hence $(x,j') \in Y$ for almost every $(x,j') \in X$, which gives Axiom (iv) for $X'$ as required.
\end{proof}

We let $M$ be a large natural number, depending on all previous quantities (in particular, depending on $\kappa$), to be chosen later.  The functions $\tilde f'_n \in L^1(\tilde X')$ will be defined for negative $n$ by the formulae
$$ \tilde f'_n(x,1,s) := \tilde f_{n}(x,s)$$
and
$$ \tilde f'_n(x,2,s) := \left(1 - \frac{\alpha}{2}\right) \tilde f_{n-2M}(x,s)$$
for any $x \in X$ and $s \in \{a,b,a^{-1},b^{-1}\}$.  Informally, $\tilde f'_n$ is two copies of $\tilde f'_n$, one over $X \times \{1\}$ and one over $X \times \{2\}$, with the latter experiencing a significant time delay and also a slight reduction in amplitude; the point is that we can delay the $X \times \{2\}$ dynamics until the dynamics of $X \times \{1\}$ has mixed almost completely, so that half of the mass of the $X \times \{1\}$ component is spread out almost uniformly over $X \times \{2\}$, allowing for the crucial amplitude reduction for the $X \times \{2\}$ component.  The idea behind this construction is due to Ornstein \cite[Lemma 4]{ornstein}.

Clearly, Axiom (vii') is a consequence of Axiom (vii) (we allow the constant $A'$ to depend on $M$).  For functions supported on $\tilde X'_0$, the Markov operator $P'$ is a trivial lift of the Markov operator $P$, so (from Axiom (vii')) one sees that $\tilde f'_{n+1} = P' \tilde f'_n$ for all $n \leq -2$.  We now define $\tilde f'_n$ for non-negative $n$ by setting
$$ \tilde f'_n := (P')^{n+1} \tilde f'_{-1},$$
so that Axiom (v') holds.  Clearly the $\tilde f'_n$ are non-negative and in $L^\infty$, and direct calculation shows that Axiom (vi') holds for all negative $n$, and hence for all $n$ thanks to Axiom (v').

The only remaining task is to show Axiom (viii').  By the union bound, it suffices to show the bounds on $X \times \{1\}$ and $X \times \{2\}$ separately.  More precisely, we establish the following two propositions.

\begin{proposition}\label{lima}  If $\kappa$ is sufficiently small (depending on $\eps, N$, and the $\tilde f_n$, but without any dependence on $M$), we have
$$ \sup_{n \in \Z} \pi_* \tilde f'_{2n}(x,1) \geq 1-\eps$$
for all $x \in X$ outside of a set of measure at most $\eps$.
\end{proposition}

\begin{proof}
By construction, we have
$$ \tilde f'_n(x,1,s) = \tilde f_n(x,s) $$
for negative $n$, all $x \in X$, and $s \in \{a,b,a^{-1},b^{-1}\}$.  Now we turn to non-negative $n$.  Note that as $P$ is a contraction on $L^\infty$, the $\tilde f_n$ for non-negative $n$ are uniformly bounded in $L^\infty$ by some quantity $B$ independent of $\kappa$.  A routine induction then shows that
$$ \int_{\tilde X} \max( \tilde f_n(x,s) - \tilde f'_n(x,1,s), 0 )\ d\tilde \mu(x,s) \leq C_{B,n} \kappa$$
for all non-negative $n$ and some quantity $C_{B,n}$ that depends on $B,n$ but not on $\kappa$; this is basically because on $X \times \{1\} \times \{a,b,a^{-1},b^{-1}\}$, the Markov process associated to $P'$ only differs from that associated to $P$ on the set $E \times \{1\} \times \{b\} \cup T_b E \times \{1\} \cup \{b^{-1}\}$, which has measure $\kappa/2$.  Applying $\pi_*$ and then the triangle inequality, we conclude that
$$
 \int_{X} \max\left( \sup_{-N \leq n \leq N} \pi_* \tilde f_{2n}(x) - \sup_{-N \leq n \leq N} \pi_* \tilde f'_{2n}(x,1), 0 \right)\ d\mu(x) \leq C'_{B,N} \kappa $$
for some $C'_{B,N}$ independent of $\kappa$; in particular, from Markov's inequality we see (for $\kappa$ small enough) that
$$ \sup_{-N \leq n \leq N} \pi_* \tilde f_{2n}(x) - \sup_{-N \leq n \leq N} \pi_* \tilde f'_{2n}(x,1) \leq \eps/3$$
for all $x \in X$ outside of a set of measure at most $\eps/3$.  
Combining this with \eqref{proj}, we obtain the claim.
\end{proof}

\begin{proposition}  If $\kappa$ is sufficiently small (depending on $\eps, N$, and the $\tilde f_n$, but without any dependence on $M$), we have
$$ \sup_{n \in \Z} \pi_* \tilde f'_{2n}(x,2) \geq 1-\eps$$
for all $x \in X$ outside of a set of measure at most $\eps$.
\end{proposition}

\begin{proof}  We split
$$ \tilde f'_n = \tilde f'_{n,1} + \tilde f'_{n,2}$$
where for negative $n$, $\tilde f'_{n,i}$ is the restriction of $\tilde f'_n$ to $X \times \{i\} \times \{a,b,a^{-1},b^{-1}\}$, and for non-negative $n$, $\tilde f'_{n,i}$ is propagated by $P'$:
$$ \tilde f'_{n,i} := (P')^{n+1} \tilde f'_{-1,i}.$$
Observe that the $\tilde f'_{n,1}$ component of $\tilde f'_n$ does not depend on $M$.

From Lemma \ref{pet-good}, we see that $\tilde f'_{n,1}$ converges pointwise almost everywhere as $n \to \infty$ to the constant
$$ \frac{1}{\mu(\tilde X')} \int_{\tilde X'} \tilde f'_{-1,1}\ d\tilde \mu' = \frac{1}{2} \int_X \tilde f_{-1}\ d\tilde \mu = \frac{\alpha}{2}.$$
In particular, $\pi_* \tilde f'_{n,1}$ converges pointwise almost everywhere to the same constant.
Thus, by Egorov's theorem, and assuming $M$ sufficiently large (depending on previous quantities such as $\eps, \kappa$, and the $\tilde f_n$, but without any circular dependency of $M$ on itself) we have
\begin{equation}\label{add}
\inf_{n \geq 2M-2N} \pi_* \tilde f'_{n,1}(x,2) \geq \frac{\alpha}{2} - \frac{\eps}{3}
\end{equation}
for all $x \in X$ outside of a set of measure at most $\eps/3$.

Now we work on $\tilde f'_{n,2}$.  For all $n < 2M$, an induction (using Axiom (vii)) shows that $\tilde f'_{n,2}$ is supported on $X_0 \times \{2\} \times \{a,b,a^{-1},b^{-1}\}$, and that
$$ \tilde f'_{n,2}(x,2,s) = \left(1 - \frac{\alpha}{2}\right) \tilde f_{n-2M}(x,s)$$
for all $x \in X$ and $s \in \{a,b,a^{-1},b^{-1}\}$.  Repeating the arguments used to prove Proposition \ref{lima}, we see (if $\kappa$ is sufficiently small depending on $\eps,N$, but (crucially) without any dependence on $M$) that 
$$ \left(1 - \frac{\alpha}{2}\right) \sup_{M-N \leq n \leq M+N} \pi_* \tilde f_{2n-2M}(x)
-\sup_{M-N \leq n \leq M+N} \pi_* \tilde f'_{2n}(x,2) \leq \eps/3
$$
for all $x \in X$ outside of a set of measure at most $\eps/3$.
Combining this with \eqref{add}, we see that
$$
\sup_n \pi_* \tilde f'_{n}(x,2) \geq \frac{\alpha}{2} + (1 - \frac{\alpha}{2}) \sup_{-N \leq n \leq N} \pi_* \tilde f_{2n}(x) - \frac{2\eps}{3} $$
for all $x \in X$ outside of a set of measure at most $2\eps/3$.  Applying \eqref{proj}, we then obtain the claim.
\end{proof}

The proof of Theorem \ref{iterate}, and thus Theorem \ref{main}, is now complete.

\end{document}